\documentclass[11pt,bezier]{article}
\usepackage{amsmath,graphicx,amssymb,amsfonts}
\usepackage{color}
\usepackage[normalem]{ulem}

\newtheorem{thm}{Theorem}[section] 
\newtheorem{lem}{Lemma}[section] 
\newtheorem{cons}{Construction}[section] 
\newtheorem{ex}{Example}[section] 
 
\newtheorem{remark}{Remark}[section]

\newtheorem{preproof}{{\bf Proof.}}

\newenvironment{proof}[1]{\begin{preproof}{\rm
               #1}\hfill{$\Box$}}{\end{preproof}}

\title{\bf\Large Some Constructions for Amicable Orthogonal Designs}
\author{
Ebrahim Ghaderpour $^{1}$ \\[0.3cm]
{\sl Department of Earth and Space Science and Engineering}\\
{\sl York University, Toronto, ON, Canada M3J-1P3}\\[0.2cm]}

\date{}

\begin{document}\sloppy
\maketitle
\footnotetext[1]{E-mail: {\tt
ebig2@yorku.ca.}}

\begin{abstract}
\noindent
Hadamard matrices, orthogonal designs and amicable orthogonal designs have a number of applications in coding theory, cryptography, wireless network communication and so on. 
Product designs were introduced by Robinson in order to construct orthogonal designs especially full orthogonal designs (no zero entries) with maximum number of variables for some orders. 
He constructed product designs of orders $4$, $8$ and $12$ and types  $\big(1_{(3)};  1_{(3)};  1\big),$ $\big(1_{(3)};  1_{(3)};  5\big)$ and $\big(1_{(3)};  1_{(3)};  9\big)$, respectively.
In this paper, we first show that there does not exist any product design of order $n\neq 4$, $8$, $12$ and type
$\big(1_{(3)};  1_{(3)};  n-3\big),$ where the notation $u_{(k)}$ is used to show that $u$ repeats $k$ times.  Then, following the Holzmann and Kharaghani's methods,
we construct some classes of disjoint and  some classes of full amicable orthogonal designs,
and we obtain an infinite class of full amicable orthogonal designs.  Moreover, a full amicable orthogonal design of order $2^9$ and type $\big(2^6_{(8)}; 2^6_{(8)}\big)$ is constructed.

\vspace{.7cm}
 \end{abstract}
\section{Introduction}\label{introduction}
The definitions in this section can be all found in \cite{GS}.

A {\it Hadamard matrix} of order $n$ is a square matrix of order $n$ with $\pm 1$ entries such that $$HH^{\rm T}=nI_n,$$ where $H^{\rm T}$ is the transpose of $H$, and $I_n$ is the identity matrix of order $n$. It is conjectured that a Hadamard matrix of order $4m$ exists for each $m\ge 1$.

An {\it orthogonal design} (OD) of order $n$ and type $(c_1, \ldots, c_k)$, denoted $OD(n; \ c_1, \ldots, c_k)$, 
 is a square matrix $C$ of order $n$ with entries from $\{0, \pm  x_1, \ldots, \pm  x_k\}$  that satisfies
$$ CC^{\rm T}= \Big(\sum_{j=1}^k c_j x_j^2 \Big)I_n,$$ where the $x_j$'s are commuting variables.  
An OD with no zero entry is called a {\it full} OD.   A Hadamard matrix can be obtained by equating all variables of a full OD to 1.

Radon \cite{Radon} worked on a proposition concerning the composition of quadratic forms (extended by  Hurwitz \cite{Hurwitz}) which has a connection in determination of the  maximum number of variables in ODs. According to the proposition, the maximum number of variables in an OD of order $n=2^ab,$ $b$ odd, is $\rho(n)=8c+2^d$, where $a=4c+d$, $0\le d<4$  (see \cite[Chapter 1]{GS}). This number is called {\it Radon-Hurwitz number}.

Two square matrices $A$ and $B$ are called {\it amicable} if $AB^{\rm T}=BA^{\rm T}$. They are called {\it anti-amicable} if $AB^{\rm T}=-BA^{\rm T}$.
Suppose that $C$ is an $OD(n; \ c_1,c_2,\ldots,c_k)$ with variables $x_1,\ldots, x_k$, and $D$ is an $OD(n; \ d_1,d_2,\ldots,d_m)$ with variables $y_1,\ldots, y_m$, where the sets  $\{x_1,\ldots, x_k\}$ and $\{y_1,\ldots, y_m\}$ are disjoint. Then $(C;D)$ is called an {\it amicable orthogonal design} (AOD)
denoted $$AOD\big( n;\ c_1,c_2,\ldots,c_k; \ d_1,d_2,\ldots,d_m\big),$$ if $CD^{\rm T}=DC^{\rm T}$.   It can be seen that if $(C;D)$ is an AOD, then $\left[ \begin {array}{rr} C&D\\ D&-C\end {array} \right]$ forms an $$OD\big( 2n;\ c_1,c_2,\ldots,c_k, d_1,d_2,\ldots,d_m\big).$$

Wolfe \cite{Wolfe} showed that the total number of
variables in an AOD of order $n=2^ab,$ $b$ odd, is less than or equal to $2a+2$.

A {\it rational family} of order $n$ and type $(r_1 , \ldots , r_k),$ where the $r_j$'s are positive rational numbers, is a collection of $k$ 
rational matrices of order $n,$ $A_1, \ldots, A_k,$ satisfying

(i)  $A_iA_i^{\rm T}=r_iI_n,   \ \ 1\leq i \leq k;$

(ii)  $A_iA_j^{\rm T}=-A_jA_i^{\rm T}, \ \ 1\leq i\neq j \leq k.$ 

The {\it Hadamard product} of two square matrices $A=[a_{ij}]$ and $B=[b_{ij}]$ of order $n$, denoted $A*B$, is a square matrix of order $n$ such that its entries are computed via entrywise multiplication of $A$ and $B$, i.e., $A*B=[a_{ij}b_{ij}]$. $A$ and $B$ are called {\it disjoint} if $A*B={\bf0}$.

Let $A=\big(a_1, \ldots, a_{n}\big).$ The square matrix $C=[c_{ij}]$ of order $n$ is called {\it circulant} if $c_{ij}=a_{j-i+1}$, denoted ${\rm circ}\big(a_1, \ldots, a_{n}\big)$, where $j-i$ is reduced modulo $n$. The square matrix  $B=[b_{ij}]$ of order $n$ is called {\it back-circulant} if $b_{ij}=a_{i+j-1}$, denoted ${\rm backcirc}\big(a_1, \ldots, a_{n}\big)$, where $i+j-2$ is reduced modulo $n$. 
It is shown that (see \cite[Chapter 4]{GS}) if $B$ is  back-circulant and $A$ and $C$ are circulant matrices of order $n$, then
$B=B^{\rm T}$, $AC=CA$ and $BC^{\rm T}=CB^{\rm T}$.

Assume that $M_1$, $M_2$ and $N$ are ODs of order $n$ and types $(a_1,\ldots, a_r),$  $(b_1,\ldots , b_s)$ and $(u_1, \ldots, u_t),$ respectively.
Then $(M_1; M_2; N)$ is called a {\it product design} (PD) of order $n$ and type $\big(a_1,\ldots, a_r;  b_1, \ldots, b_s;  u_1,\ldots, u_t\big)$, denoted $PD\big(n; \  a_1,\ldots, a_r;  b_1, \ldots, b_s;  u_1,\ldots, u_t \big),$ if the following conditions hold:

(i)  $M_1*N=M_2*N={\bf0}, $ 

(ii)   $M_1+N $ and $M_2+N$ are ODs, and 

(iii)   $M_1M_2^{\rm T}=M_2M_1^{\rm T}.$ 

The following theorem from Robinson \cite{Robinson} shows how to construct an OD by combining
an AOD and a PD.

\begin{thm} \label{pdod}
Suppose $(C; D_1+D_2)$ is an $AOD\big(m; \ c_1, \ldots, c_k; \ v,w_1,\ldots, w_{\ell}\big)$ and  $(M_1;M_2;N)$ 
is a $PD\big(n; \ a_1,\ldots, a_r;  b_1,\ldots, b_s; u_1,\ldots, u_t\big),$ where $D_1$ is an $OD(m;\ v)$ and $D_2$ is an $OD\big(m;\ w_1,\ldots, w_{\ell}\big)$. Let 
$b$, $c$, $u$ and $w$ be the sums of the $b_i$'s, $c_i$'s, $u_i$'s and $w_i$'s, respectively.
Then there exist

{\rm (i)}  $OD\big(mn; \ va_1,\ldots, va_r, wb_1, \ldots, wb_s, cu_1,\ldots, cu_t\big),$ 

{\rm (ii)}  $OD\big(mn; \ va_1,\ldots, va_r, wb_1, \ldots, wb_s, c_1u,\ldots, c_ku\big),$ 

{\rm (iii)}  $OD\big(mn; \ va_1,\ldots, va_r, w_1b, \ldots, w_{\ell}b, cu_1,\ldots, cu_t\big),$ 

{\rm (iv)}  $OD\big(mn; \ va_1,\ldots, va_r, w_1b, \ldots, w_{\ell}b, c_1u,\ldots, c_ku\big).$
\end{thm}

Robinson \cite{Robinson} constructed product designs
$PD\big(4; \ 1_{(3)};  1_{(3)};  1\big),$ $PD\big(8; \ 1_{(3)};  1_{(3)};  5\big)$ and $PD\big(12; \ 1_{(3)};  1_{(3)}; 9\big).$ He used these PDs and applied Theorem \ref{pdod} with some known AODs to construct some full ODs of small orders with maximum number of variables. For instance, he applied Theorem \ref{pdod} to a $PD\big(12; \ 1_{(3)};  1_{(3)}; 9\big)$ and an $AOD\big(2; \ 1_{(2)};1_{(2)}\big)$ to construct an $OD\big(24; 1_{(6)},9_{(2)}\big)$.

\section{A non-existence result for product designs}
Although Robinson constructed $PD\big(4; \ 1_{(3)};  1_{(3)};  1\big),$ $PD\big(8; \ 1_{(3)};  1_{(3)};  5\big)$, $PD\big(12; \ 1_{(3)};  1_{(3)}; 9\big)$, he did not show if there is any $PD\big(n; \ 1_{(3)};  1_{(3)};  n-3\big)$ for some $n\neq 4,8,12$. In this section, we show that in fact there does not exist any  $PD\big(n; \ 1_{(3)};  1_{(3)};  n-3\big)$ for all $n\neq 4,8,12.$ In doing so, we first mention
the following well known theorems, and then we prove Theorem \ref{pdnon}.

\begin{thm} {\rm (Vinogradov \cite{vin}).}\label{Hilbert}
Suppose that  $a$, $a'$, $b$ and $c$ are nonzero p-adic numbers, $p$ is a prime number, and $r$ and $s$ are positive integers. Define  $(a,b)_p$, the $p$-adic Hilbert symbol, to be $1$ if there are p-adic numbers $x$ and $y$ such that $ax^2+by^2=1$, and $-1$ otherwise. Then

{\rm (i)}   $(a,b)_p=(b,a)_p, \ (a,c^2)_p=1,$

{\rm (ii)}  $(a,-a)_p=1, \ (a,1-a)_p=1,$

{\rm (iii)}  $(aa',b)_p=(a,b)_p(a',b)_p,$ 

\noindent and if $p\neq 2$, then

{\rm (iv)}  $(r,s)_p=1$ if $r$ and $s$ are relatively prime to $p,$

{\rm (v)}   $(r,p)_p=(r/p),$ the Legendre symbol, if $r$ and $p$ are relatively prime,

{\rm (vi)}  $(p,p)_p=(-1/p).$
\end{thm}

\begin{thm} {\rm (Shapiro \cite{Shapiro}).}\label{rational9} 
 There is a rational family of type $(s_1, \ldots, s_9)$ and order $16$ if and only if 
$S_p(s_1, \ldots, s_9):=\displaystyle\prod_{i<j}(s_i, s_j)_p=1 $ for every prime $p.$ 
 \end{thm}

\begin{thm} {\rm (Shapiro \cite{Shapiro}).}\label{rat}
 There is a rational family of type $(s_1, \ldots, s_k)$ and order $2^ab$,  $b$ odd, if and only if there is a rational family
of the same type and order $2^a.$
\end{thm}

\begin{thm} {\rm (Robinson \cite{Robnon}).}\label{nonexistence}
 There does not exist any $OD\big(n; \ 1_{(5)},n-5\big)$ for $ n>40.$
\end{thm}

\begin{thm} {\rm (Kharaghani and  Tayfeh-Rezaie \cite{kht}).}\label{od32}
 There is a full $OD\big(32; \ 1_{(5)}, u_1, \ldots, u_k\big)$ if and only if $(u_1,\ldots, u_k) = (9, 9, 9)$ or $(9,18)$ or $(12, 15)$ or $(27)$.
\end{thm}

\begin{thm}\label{pdnon}
 There does not exist any $PD\big(n;  \ 1_{(3)}; 1_{(3)};  n-3\big)$ for $n\neq 4,8,12.$
\end{thm}
\begin{proof}
{If there exists a $PD\big(n; \ 1_{(3)}; 1_{(3)}; n-3\big) $ for some $ n>20,$ then applying Theorem \ref{pdod} with an $AOD\big(2;  1_{(2)}; 1_{(2)}\big)$ gives an $OD\big(2n; \ 1_{(6)},2n-6\big)$ 
which contradicts Theorem \ref{nonexistence}. 
From the definition of PD, there is no $PD\big(n; \ 1_{(3)}; 1_{(3)}; n-3\big)$ for $n=1$, $2$, $3$, $5$, $6$, $7$, $9$, $10$, $11$, $13$, $14$, $15$, $17$, $18$ and $19$. If $n=16,$ then from the above argument, there exists an $OD\big(32; \ 1_{(6)},26\big)$ which is impossible by Theorem \ref{od32}. Thus, there is no $PD\big(16; \ 1_{(3)}; 1_{(3)}; 13\big)$.
Now suppose that there exists a $PD\big(20; \ 1_{(3)};  1_{(3)};  17\big).$ 
Applying Theorem \ref{pdod} to this PD and an $AOD\big(4; \ 1_{(2)},2;  1_{(2)},2\big)$ (see \cite[Chapter 5]{GS} for the existence of this AOD) gives an $OD(80; 1_{(3)},3_{(3)},17_{(2)},34).$ Hence, by Theorem \ref{rat},
 there is a rational family  of type $(1_{(3)},3_{(3)},17_{(2)},34)$ and order $16$.
It can be seen that Theorem \ref{Hilbert}  gives $S_{17}(1_{(3)},3_{(3)},17_{(2)},34)=-1$ which contradicts Theorem \ref{rational9}. Therefore, there does not exist any $PD\big(20;  \ 1_{(3)}; 1_{(3)};  17\big)$.} 
\end{proof}

\section{Some full amicable orthogonal designs}
In this section, we combine techniques similar to \cite{KHH} and \cite{Robinson} to obtain some classes of full amicable orthogonal designs. 
\begin{cons}{\label{construction1}\rm
Suppose that $A_1$, $A_2$, $B$, $C$, $D$, $E$, $F$, $G$ are square matrices of order $n$, and ${\bf0}$ is the zero matrix. Let
$$M_1=\left[\begin{array}{rrrr|rrrr}
{\bf0}&D&B&C&{\bf0}&{\bf0}&{\bf0}&{\bf0}\\
-D&{\bf0}&-C&B&{\bf0}&{\bf0}&{\bf0}&{\bf0}\\
B&-C&{\bf0}&D&{\bf0}&{\bf0}&{\bf0}&{\bf0}\\
C&B&-D&{\bf0}&{\bf0}&{\bf0}&{\bf0}&{\bf0}\\
\hline
{\bf0}&{\bf0}&{\bf0}&{\bf0}&{\bf0}&-D&B&C\\
{\bf0}&{\bf0}&{\bf0}&{\bf0}&D&{\bf0}&-C&B\\
{\bf0}&{\bf0}&{\bf0}&{\bf0}&B&-C&{\bf0}&-D\\
{\bf0}&{\bf0}&{\bf0}&{\bf0}&C&B&D&{\bf0}
\end{array}\right],$$

$$ M_2=\left[\begin{array}{rrrr|rrrr}
{\bf0}&G&E&F&{\bf0}&{\bf0}&{\bf0}&{\bf0}\\
-G&{\bf0}&F&-E&{\bf0}&{\bf0}&{\bf0}&{\bf0}\\
E&F&{\bf0}&-G&{\bf0}&{\bf0}&{\bf0}&{\bf0}\\
F&-E&G&{\bf0}&{\bf0}&{\bf0}&{\bf0}&{\bf0}\\
\hline
{\bf0}&{\bf0}&{\bf0}&{\bf0}&{\bf0}&-E&F&G\\
{\bf0}&{\bf0}&{\bf0}&{\bf0}&E&{\bf0}&G&-F\\
{\bf0}&{\bf0}&{\bf0}&{\bf0}&F&G&{\bf0}&E\\
{\bf0}&{\bf0}&{\bf0}&{\bf0}&G&-F&-E&{\bf0}
\end{array}\right], $$

$$
N_i=\left[\begin{array}{rrrr|rrrr}
A_i&{\bf0}&{\bf0}&{\bf0}&A_i&-A_i&A_i&A_i\\
{\bf0}&A_i&{\bf0}&{\bf0}&A_i&A_i&A_i&-A_i\\
{\bf0}&{\bf0}&-A_i&{\bf0}&-A_i&-A_i&A_i&-A_i\\
{\bf0}&{\bf0}&{\bf0}&-A_i&-A_i&A_i&A_i&A_i\\
\hline
-A_i&-A_i&-A_i&-A_i&A_i&{\bf0}&{\bf0}&{\bf0}\\
A_i&-A_i&-A_i&A_i&{\bf0}&A_i&{\bf0}&{\bf0}\\
A_i&A_i&-A_i&-A_i&{\bf0}&{\bf0}&-A_i&{\bf0}\\
A_i&-A_i&A_i&-A_i&{\bf0}&{\bf0}&{\bf0}&-A_i
\end{array}\right],$$

\noindent for $i\in\{1,2\}.$ Now suppose that  $A_1, A_2, B, C, D, E, F$ and $ G$ are pairwise amicable (not necessary ODs),
and they satisfy the following properties:
\begin{align}\label{5a}
&5A_1A_1^{\rm T}+BB^{\rm T}+CC^{\rm T}+DD^{\rm T}=kI_n, \\ \label{5b}
&5A_2A_2^{\rm T}+EE^{\rm T}+FF^{\rm T}+GG^{\rm T}=sI_n,
\end{align}
where $k$ and $s$ are two quadratic forms. Let 
\begin{align}\label{IPQR}
I:=\left[ \begin {array}{rr} 1&0\\ 0&1\end {array} \right]\!, \ \  P:= \left[ \begin {array}{rr} 0&1\\ 1&0\end {array} \right]\!,  \ \ Q:= \left[ \begin {array}{rr} 1&0\\ 0&-1\end {array} \right]\!, \ \ R:= \left[ \begin {array}{rr} 0&1\\ -1&0\end {array} \right]\!.
\end{align}
Note that $PQ^{\rm T}=-QP^{\rm T}$, $RI^{\rm T}=-IR^{\rm T}$,  $PR^{\rm T}=RP^{\rm T}$, $PI^{\rm T}=IP^{\rm T}$, $QI^{\rm T}=IQ^{\rm T}$, $RQ^{\rm T}=QR^{\rm T}$, $N_1N_2^{\rm T}=N_2N_1^{\rm T}$, $M_1M_2^{\rm T}=M_2M_1^{\rm T}$ and for $i, j\in \{1,2\}$, $M_jN_i^{\rm T}=-N_iM_j^{\rm T}$. Also, $I*P={\bf0}$, $Q*R={\bf0}$, $M_j*N_i={\bf0}$  for $i,j\in \{1,2\}$. Using Equations \eqref{5a} and \eqref{5b} and the above properties, it can be verified that the following matrices form an AOD of order $16n$:
\begin{align}\label{dis}
 U=N_1\otimes I+M_1\otimes Q, \ \ \ \  V=N_2\otimes P+M_2\otimes R^{\rm T},
\end{align}
where $\otimes$ is the Kronecker product. Moreover, it can be seen that $U$ and $V$ in Equation \eqref{dis}, are disjoint.

Assume that the matrices $A_1$, $A_2$, $B$, $C$, $D$, $E$, $F$ and $G$ are 
full (no zero entries) pairwise amicable, and $H$ is a Hadamard matrix of order $2$. Then the following matrices form a full AOD of order $16n$:
\begin{align}\label{amicable1}
 U_H=N_1\otimes H+M_1\otimes QH, \ \ \ \ V_H=N_2\otimes PH+M_2\otimes R^{\rm T}H.
\end{align}}
\end{cons}

\begin{ex}{\label{aodfull48} \rm
Consider $A_1={\rm backcirc}(x,-b,b)$, $A_2={\rm circ}(-d,d,d)$, $B={\rm circ}(b,b,b)$, $C={\rm circ}(-a,b,b)$, $D={\rm circ}(a,b,b)$, $E={\rm circ}(d,d,d)$, $F={\rm circ}(-c,d,d)$ and  $G={\rm circ}(c,d,d).$ 
It can be seen that the conditions of Construction \ref{construction1} hold, and so matrices $U_H$ and $V_H$ given by Equation \eqref{amicable1} 
are $AOD\big(48;  \ 4,10,34; \ 4,44\big)$ (see the appendix in \cite{Ghaderpour}).
}
\end{ex}

 
\begin{cons}{\label{construction2}\rm
Suppose that $A_1$, $A_2$, $B$, $C$, $D$, $E$, $F$, $G$ are square matrices of order $n$, and ${\bf0}$ is the zero matrix. Let
$$ M_1=\left[\begin{array}{rrrr|rrrr|rrrr}
B&C&D&{\bf0}&{\bf0}&{\bf0}&{\bf0}&{\bf0}&{\bf0}&{\bf0}&{\bf0}&{\bf0}\\
-C&B&{\bf0}&-D&{\bf0}&{\bf0}&{\bf0}&{\bf0}&{\bf0}&{\bf0}&{\bf0}&{\bf0}\\
-D&{\bf0}&B&C&{\bf0}&{\bf0}&{\bf0}&{\bf0}&{\bf0}&{\bf0}&{\bf0}&{\bf0}\\
{\bf0}&D&-C&B&{\bf0}&{\bf0}&{\bf0}&{\bf0}&{\bf0}&{\bf0}&{\bf0}&{\bf0}\\
\hline
{\bf0}&{\bf0}&{\bf0}&{\bf0}&B&C&D&{\bf0}&{\bf0}&{\bf0}&{\bf0}&{\bf0}\\
{\bf0}&{\bf0}&{\bf0}&{\bf0}&-C&B&{\bf0}&-D&{\bf0}&{\bf0}&{\bf0}&{\bf0}\\
{\bf0}&{\bf0}&{\bf0}&{\bf0}&-D&{\bf0}&B&C&{\bf0}&{\bf0}&{\bf0}&{\bf0}\\
{\bf0}&{\bf0}&{\bf0}&{\bf0}&{\bf0}&D&-C&B&{\bf0}&{\bf0}&{\bf0}&{\bf0}\\
\hline
{\bf0}&{\bf0}&{\bf0}&{\bf0}&{\bf0}&{\bf0}&{\bf0}&{\bf0}&B&C&D&{\bf0}\\
{\bf0}&{\bf0}&{\bf0}&{\bf0}&{\bf0}&{\bf0}&{\bf0}&{\bf0}&-C&B&{\bf0}&-D\\
{\bf0}&{\bf0}&{\bf0}&{\bf0}&{\bf0}&{\bf0}&{\bf0}&{\bf0}&-D&{\bf0}&B&C\\
{\bf0}&{\bf0}&{\bf0}&{\bf0}&{\bf0}&{\bf0}&{\bf0}&{\bf0}&{\bf0}&D&-C&B
\end{array}\right], $$

$$  M_2=\left[\begin{array}{rrrr|rrrr|rrrr}
E&F&G&{\bf0}&{\bf0}&{\bf0}&{\bf0}&{\bf0}&{\bf0}&{\bf0}&{\bf0}&{\bf0}\\
F&-E&{\bf0}&-G&{\bf0}&{\bf0}&{\bf0}&{\bf0}&{\bf0}&{\bf0}&{\bf0}&{\bf0}\\
G&{\bf0}&-E&F&{\bf0}&{\bf0}&{\bf0}&{\bf0}&{\bf0}&{\bf0}&{\bf0}&{\bf0}\\
{\bf0}&-G&F&E&{\bf0}&{\bf0}&{\bf0}&{\bf0}&{\bf0}&{\bf0}&{\bf0}&{\bf0}\\
\hline
{\bf0}&{\bf0}&{\bf0}&{\bf0}&F&G&E&{\bf0}&{\bf0}&{\bf0}&{\bf0}&{\bf0}\\
{\bf0}&{\bf0}&{\bf0}&{\bf0}&G&-F&{\bf0}&-E&{\bf0}&{\bf0}&{\bf0}&{\bf0}\\
{\bf0}&{\bf0}&{\bf0}&{\bf0}&E&{\bf0}&-F&G&{\bf0}&{\bf0}&{\bf0}&{\bf0}\\
{\bf0}&{\bf0}&{\bf0}&{\bf0}&{\bf0}&-E&G&F&{\bf0}&{\bf0}&{\bf0}&{\bf0}\\
\hline
{\bf0}&{\bf0}&{\bf0}&{\bf0}&{\bf0}&{\bf0}&{\bf0}&{\bf0}&G&-E&-F&{\bf0}\\
{\bf0}&{\bf0}&{\bf0}&{\bf0}&{\bf0}&{\bf0}&{\bf0}&{\bf0}&-E&-G&{\bf0}&F\\
{\bf0}&{\bf0}&{\bf0}&{\bf0}&{\bf0}&{\bf0}&{\bf0}&{\bf0}&-F&{\bf0}&-G&-E\\
{\bf0}&{\bf0}&{\bf0}&{\bf0}&{\bf0}&{\bf0}&{\bf0}&{\bf0}&{\bf0}&F&-E&G
\end{array}\right],$$

$${ \small N_i=\left[\begin{array}{rrrr|rrrr|rrrr}
{\bf0}&{\bf0}&{\bf0}&-A_i&A_i&A_i&A_i&-A_i&A_i&-A_i&A_i&-A_i\\
{\bf0}&{\bf0}&-A_i&{\bf0}&A_i&-A_i&-A_i&-A_i&-A_i&-A_i&-A_i&-A_i\\
{\bf0}&A_i&{\bf0}&{\bf0}&A_i&A_i&-A_i&A_i&A_i&A_i&-A_i&-A_i\\
A_i&{\bf0}&{\bf0}&{\bf0}&A_i&-A_i&A_i&A_i&A_i&-A_i&-A_i&A_i\\
\hline
-A_i&-A_i&-A_i&-A_i&{\bf0}&{\bf0}&{\bf0}&-A_i&A_i&A_i&-A_i&A_i\\
-A_i&A_i&-A_i&A_i&{\bf0}&{\bf0}&-A_i&{\bf0}&A_i&-A_i&A_i&A_i\\
-A_i&A_i&A_i&-A_i&{\bf0}&A_i&{\bf0}&{\bf0}&-A_i&-A_i&-A_i&A_i\\
A_i&A_i&-A_i&-A_i&A_i&{\bf0}&{\bf0}&{\bf0}&-A_i&A_i&A_i&A_i\\
\hline
-A_i&A_i&-A_i&-A_i&-A_i&-A_i&A_i&A_i&{\bf0}&{\bf0}&{\bf0}&-A_i\\
A_i&A_i&-A_i&A_i&-A_i&A_i&A_i&-A_i&{\bf0}&{\bf0}&-A_i&{\bf0}\\
-A_i&A_i&A_i&A_i&A_i&-A_i&A_i&-A_i&{\bf0}&A_i&{\bf0}&{\bf0}\\
A_i&A_i&A_i&-A_i&-A_i&-A_i&-A_i&-A_i&A_i&{\bf0}&{\bf0}&{\bf0}
\end{array}\right],}$$

\noindent for $i\in\{1,2\}.$ Suppose that $A_1$, $A_2$, $B$, $C$, $D$, $E$, $F$ and $G$ are pairwise amicable ${\rm (}$not necessarily orthogonal designs${\rm )}$
and they satisfy the following properties:
\begin{align*}
&9A_1A_1^{\rm T}+BB^{\rm T}+CC^{\rm T}+DD^{\rm T}=uI_n, \\ 
&9A_2A_2^{\rm T}+EE^{\rm T}+FF^{\rm T}+GG^{\rm T}=vI_n,
\end{align*}
where $u$ and $v$ are two quadratic forms. Then, as in Construction $\ref{construction1}$, the matrices $U$ and $V$ in Equation $\eqref{dis}$ along with these new matrices ($M_1$, $M_2$, $N_1$, $N_2$) form an AOD of order $24n$; and $U$ and $V$ are disjoint.  Moreover, matrices $U_H$ and $V_H$ in Equation $\eqref{amicable1}$ along with these new matrices ($M_1$, $M_2$, $N_1$, $N_2$) form a full AOD of order $24n,$ provided the matrices $A_1$, $A_2$, $B$, $C$, $D$, $E$, $F$ and $G$ in this construction have no zero entries.}
\end{cons}

In the following two examples, we construct two full AODs using the above construction, and we refer to the appendix in \cite{Ghaderpour} for the display of these AODs.
\begin{ex}
{\rm Suppose that 
 $A_1={\rm backcirc}(x,-b,b)$, $A_2={\rm circ}(-d,d,d)$, $B={\rm circ}(b,b,b)$, $C={\rm circ}(b,b,b)$, $D={\rm circ}(b,b,b)$, 
$E={\rm circ}(d,d,d)$, $F={\rm circ}(d,d,d)$ and $G={\rm circ}(d,d,d).$ Then they satisfy all the conditions of 
Construction \ref{construction2}, and so matrices $U_H$ and $V_H$ given by Equation \eqref{amicable1}
form a full $ AOD\big(72; \ 18,54; \ 72\big).$}
\end{ex}
\begin{ex}
{\rm Consider $A_1={\rm backcirc}(a,-a,-a,a,-a,a,a)$,  $A_2={\rm circ}(c,-c,-c,c,-c,c,c)$, $B={\rm circ}(b,a,a,a,a,a,a)$, 
$C={\rm circ}(-b,a,a,a,a,a,a)$, $D={\rm circ}(a,-a,-a,a,-a,a,a)$, $E={\rm circ}(d,c,c,c,c,c,c)$, 
$F={\rm circ}(-d,c,c,c,c,c,c)$ and $G={\rm circ}(c,-c,-c,c,-c,c,c)$. 
It can be verified that these matrices satisfy all the conditions of 
Construction \ref{construction2}, and so matrices $U_H$ and $V_H$ given by Equation \eqref{amicable1} form a full $ \ AOD\big(168; \ 4,164; \ 4,164\big)$.}
\end{ex}

\begin{remark}
{\rm If we replace matrices $A_1$, $B$, $C$, $D$, $E$, $F$ and $G$  by variables in Constructions  \ref{construction1} and \ref{construction2},
then matrices $M_1$, $ M_2$ and $N_1$ will form product designs $PD\big(8; \ 1,1,1; \ 1,1,1; \ 5\big)$ and $PD\big(12; \ 1,1,1;\  1,1,1; \ 9\big)$, respectively.}
\end{remark}

\section{An infinite class of full amicable orthogonal designs}
In this section, we choose two full AODs that can be constructed from Constructions \ref{construction1} and \ref{construction2}, and show how one can obtain an infinite class of full AODs by using them. The following theorem is an application to an algebraic result that  Kawada and Iwahori \cite{kawada} obtained.
\begin{thm} {\rm (see \cite[Chapter 5]{GS}).}\label{KI}
Suppose  that $(A;B)$ is an AOD of order $n.$ Let $t$ be the  number of variables in $B$ and 
$\rho_t(n)$ be the number of variables in $A.$ 
Also,  $n=2^{4a+b}d,$  where $0\leq b<4$ and $d$ is an odd number. 
Then $$\rho_t(n)\le 8a-t+\delta +1,$$ where the values of $\delta$ are given in the following table:
\end{thm}
\begin{center}
\begin{tabular}{|c|c|c|c|c|}
 \hline
    $b$   & $0$  & $1$ & $2$ & $3$ \\
\hline
$t \equiv 0 {\pmod 4}$ & $0$ & $1$ & $3$ & $7$  \\
\hline
$t \equiv 1 {\pmod 4}$ & $ 1$ & $2$ & $3$ & $5$  \\
\hline
$t \equiv 2{\pmod 4}$ & $-1$& $3$ & $4$ & $5$  \\
\hline
$t \equiv 3{\pmod 4}$ & $-1$ & $1$ & $5$ & $6$  \\
\hline
\end{tabular}
\end{center}
\begin{thm} {\rm(Wolfe \cite{Wolfe}).}\label{infiniteaod}
Suppose that there is an $AOD\big(n; \ u_1, u_2, \ldots, u_r; v_1,v_2,\ldots, v_s\big)$. Then for each $t\ge 1$, there is an 
$$AOD\Big(2^tn; \ \ u_1, u_1, 2u_1, \ldots, 2^{t-1}u_1, 2^tu_2,\ldots, 2^tu_r; \ 2^tv_1,2^tv_2,\ldots, 2^tv_s\Big).$$
\end{thm}

\begin{cons}{\label{conaod16}\rm
 Replacing $A_1$, $B$, $C$, $D$, $A_2$, $E$, $F$ and $G$ by variables in Constructions $\ref{construction1}$ and $\ref{construction2}$, respectively, one obtains
$$AOD\big(16; \ 2,2,2,10; \ 2,2,2,10\big) \ \ \  {\rm and} \ \ \  AOD\big(24; \ 2,2,2,18; \ 2,2,2,18\big).$$ 
Applying Theorem $\ref{infiniteaod}$ for these AODs, one obtains an infinite class of full AODs:
$$AOD\Big(2^n; \ 2_{(3)}^{n-3}, 10, 10, 5\cdot 2^2,\ldots, 5\cdot 2^{n-4}; 
 \ 2_{(3)}^{n-3}, 5\cdot 2^{n-3}\Big), \ \ \ \ \ n>4,$$
$$AOD\Big(2^n\cdot 3; \ 2_{(3)}^{n-2}, 18, 18, 9\cdot 2^2,\ldots, 9\cdot 2^{n-3};  
 \ 2_{(3)}^{n-2}, 9\cdot 2^{n-2}\Big), \ \ \ n>3.$$
}
\end{cons}

\begin{ex}{\label{aod244896} \rm
From Construction \ref{conaod16}, we obtain 

(i)   $AOD\big(24; \ 2,2,2,18; \ 2,2,2,18\big),$ 

(ii)   $AOD\big(48; \ 4,4,4,18,18; \ 4,4,4,36\big),$ 

(iii) $AOD\big(96; \ 8,8,8,18,18,36; \ 8,8,8,72\big).$

\noindent According to Theorem \ref{KI}, these AODs have taken the maximum number of variables. We refer to the appendix in \cite{Ghaderpour} for the display of these AODs.
}
\end{ex}

\section{A full amicable  orthogonal designs in 16 variables}

In this section, we construct an $AOD\big(2^9; \ 2^6_{(8)}; \ 2^6_{(8)}\big)$, and consequently an $OD\big(2^{10}; \ 2^6_{(16)}\big)$.
\begin{thm} {\rm (Wolfe \cite{Wolfe}).}\label{wolfe}
 Given an integer $n=2^sd,$ where $d$ is odd and $s\geq 1,$ there exist sets $A=\big\{A_1, \ldots, A_{s+1}\big\}$ and 
$B=\big\{B_1, \ldots, B_{s+1}\big\}$ of signed permutation matrices of order $n$ such that

{\rm (i)} $A$ consists of mutually anti-amicable and disjoint matrices,

{\rm (ii)} $B$ consists of mutually anti-amicable and  disjoint matrices,

{\rm (iii)} for each $i$ and $j,$ $A_i$ and $B_j$ are amicable.
\end{thm}
\begin{proof}
{For each $k$, $2\leq k\leq s+1$, let
$$A_1=\Big(\displaystyle\otimes_{i=1}^s I\Big)\otimes I_d, \ \ \ A_k=\Big(\otimes_{i=1}^{k-2} I\Big)\otimes R\otimes \Big(\otimes_{i=k}^s P \Big)\otimes I_d,$$ 
and $$B_1=\Big(\otimes_{i=1}^s P\Big)\otimes I_d, \ \ \ B_k=\Big(\otimes_{i=1}^{k-2} I\Big)\otimes Q\otimes \Big(\otimes_{i=k}^s P\Big) \otimes I_d,$$
where $I$, $P$, $Q$, $R$ are given by Equation \eqref{IPQR}, and $I_d$ is the identity matrix of order $d$. 
Then the matrices  $A_i$'s and $B_i$'s ($1\leq i \leq s+1$) satisfy properties  ${\rm (i)}$, ${\rm (ii)}$ and ${\rm (iii)}$. }
\end{proof}

\begin{lem}\label{AOD16}
There exists an $AOD\big(2^9; \ 2^6_{(8)}; \ 2^6_{(8)}\big)$.
\end{lem}
\begin{proof}{
Suppose that $A=\{A_1,\ldots, A_8\}$ and $B=\{B_1,\ldots, B_8\}$ are two sets of signed permutation matrices of order $2^7$ satisfying properties ${\rm (i)}$, ${\rm (ii)}$ and ${\rm (iii)}$ of Theorem \ref{wolfe}. Let $H$ be a Hadamard matrix of order $2^7$. For each $j$, $1\le j\le 4$, let
$$X_j=\frac{1}{2}x_{2j-1}\big(A_{2j-1}-A_{2j}\big)H+\frac{1}{2}x_{2j}\big(A_{2j-1}+A_{2j}\big)H$$ and $$Y_j=\frac{1}{2}y_{2j-1}\big(B_{2j-1}-B_{2j}\big)H+\frac{1}{2}y_{2j}\big(B_{2j-1}+B_{2j}\big)H.$$
Note that for $1\le i\neq j\le 4$, $X_iX_j^{\rm T}=-X_jX_i^{\rm T}$ and $Y_iY_j^{\rm T}=-Y_jY_i^{\rm T}$. For $1\le i,j\le 4$, $X_iY_j^{\rm T}=Y_jX_i^{\rm T}$. 
Also, for each $j$, $1\le j\le 4$, $$X_jX_j^{\rm T}=2^6\big(x_{2j-1}^2+x_{2j}^2\big)I_{2^7}\ \ \ {\rm and} \ \ \ Y_jY_j^{\rm T}=2^6\big(y_{2j-1}^2+y_{2j}^2\big)I_{2^7}.$$
Let
\begin{align*}
C&=I \otimes I\otimes X_1+I \otimes P\otimes X_2+P \otimes I\otimes X_3+P \otimes P\otimes X_4, \\
D&=I \otimes I\otimes Y_1 \ +I \otimes P\otimes Y_2+P \otimes I\otimes Y_3 \ +P \otimes P\otimes Y_4.
\end{align*}
It can be directly verified that $\displaystyle CC^{\rm T}=\Big(2^6\sum_{i=1}^8x_i^2\Big)I_{2^9}$, $\displaystyle DD^{\rm T}=\Big(2^6\sum_{i=1}^8y_i^2\Big)I_{2^9}$ and $CD^{\rm T}=DC^{\rm T}$. Therefore, $C$ and $D$ are an $AOD\big(2^9; \ 2^6_{(8)}; \ 2^6_{(8)}\big)$.}
\end{proof}

\begin{thm}\label{128orthogonal}
There is an $OD\big(2^{10}; \ 2^6_{(16)}\big)$.
\end{thm}
\begin{proof}{
Let $C$ and $D$ be the matrices constructed in the proof of Lemma \ref{AOD16}. Since $CD^{\rm T}=DC^{\rm T}$, it can easily be verified that $\left[ \begin {array}{rr} C&D\\ D&-C\end {array} \right]$ is an $OD\big(2^{10}; \ 2^6_{(16)}\big)$.}
\end{proof}
Kharaghani \cite{Kharaghani} constructed an $OD\big(2^{10}; \ 2^6_{(16)}\big)$ using a different method. 
Since the maximum number of variables in an AOD of order $2^6$ is $14$, there does not exist any $AOD\big(2^6; \ 2^3_{(8)}; \ 2^3_{(8)}\big)$; however, it is not known whether or not there exists an $AOD\big(2^7; \ 2^4_{(8)}; \ 2^4_{(8)}\big)$. Although one of the main purposes of Robinson \cite{Robinson} for introducing PDs was to construct full ODs with the maximum number of variables, he did not construct any full OD of order $128$ with $16$ variables. It is not known whether there exists any full OD of order $128$ with the maximum number of variables. In fact, it is conjectured \cite{Kharaghani} that whether there exists an $OD\big(128; \ 2^{3}_{(16)}\big)$. 

\section{Acknowledgement}
The paper constitutes a part of the author's Ph.D. thesis written under the direction of Professor Hadi Kharaghani at the University of Lethbridge. The author would like to thank Professor Hadi Kharaghani for introducing the problem and his very useful guidance toward solving the problem and also Professor Rob Craigen for his time and great help.

\end{document}